\documentclass{amsart}

\usepackage{amsmath}
\usepackage[colorlinks=true]{hyperref}
\usepackage[english]{babel}
\usepackage{graphicx,subfigure}
\usepackage{a4wide}
\usepackage{multirow}

\newtheorem{theorem}{Theorem}[section]

\theoremstyle{definition}

\newtheorem{example}[theorem]{Example}

\title[A Finite Element Approximation for Time-Fractional Diffusion Equations]{
A Finite Element Approximation for a Class\\
of Caputo Time-Fractional Diffusion Equations}

\author[M. R. Sidi Ammi, I. Jamiai and D. F. M. Torres]{}



\email{sidiammi@ua.pt}
\email{i.jamiai@gmail.com}
\email{delfim@ua.pt}

\thanks{$^*$Corresponding author: Delfim F. M. Torres (delfim@ua.pt)}

\AtBeginDocument{{\noindent\small
This is a preprint of a paper whose final and definite form is with 
\emph{Computers and Mathematics with Applications}, ISSN 0898-1221.
Submitted 10-Oct-2018; revised 18-May-2019; accepted for publication 25-May-2019.}
\vspace{9mm}}

\begin{document}

\maketitle

\centerline{\scshape Moulay Rchid Sidi Ammi and Ismail Jamiai}
\medskip
{\footnotesize
 \centerline{AMNEA Group, Department of Mathematics,}
 \centerline{Faculty of Sciences and Technics, Moulay Ismail University,}
 \centerline{B.P. 509, Errachidia, Morocco}}

\medskip

\centerline{\scshape Delfim F. M. Torres$^*$}
\medskip
{\footnotesize
\centerline{Center for Research and Development in Mathematics and Applications (CIDMA),}
\centerline{Department of Mathematics, University of Aveiro,}
\centerline{3810-193 Aveiro, Portugal}}

\bigskip


\begin{abstract}
We develop a fully discrete scheme for time-fractional diffusion equations
by using a finite difference method in time and a finite element method in space.
The fractional derivatives are used in Caputo sense. Stability and error estimates
are derived. The accuracy and efficiency of the presented method is shown
by conducting two numerical examples.
\end{abstract}


\section{Introduction}

Fractional calculus is the field of mathematical analysis that deals with the
investigation and application of integrals and derivatives of arbitrary order.
The fractional calculus may be considered an old topic, starting from some
speculations of Leibniz and Euler, respectively in the 17th and 18th centuries,
and yet a recent subject under strong development \cite{MR3787674,MR3822307,MR3854267}.

In recent years, time-fractional partial differential equations (TFPDEs)
have aroused a considerable interest among mathematicians and also
have been applied broadly in various applications of numerical analysis
in different research areas, including fractal phenomena, diffusion processes,
complex networks, stochastic interfaces, synoptic climatology, option
pricing mechanisms, medical image processing, electromagnetic, electro-chemistry
and material sciences, and chaotic dynamics of nonlinear systems
\cite{MR3830855,MR3859787,MR3845039}.
In view of the importance of TFPDEs, many researchers investigate
them in both analytical and numerical frameworks. Several works and methods have been
developed, such as finite difference methods \cite{6,26,30,34,ZSW},
finite element methods \cite{8,11}, spectral methods \cite{XJCJ},
Adomian decomposition methods \cite{36}, and variational iteration methods \cite{10}.
Regarding analytical solutions to TFDEs, one can use
Green and Fox functions and their properties,
similarity methods, and Fourier--Laplace transforms
or Wright functions \cite{8bis,9,10bis,11bis}.

Here we study a numerical approach to the following
initial-boundary value time-fractional Caputo diffusion problem:
\begin{equation}
\label{eq:01}
\begin{split}
& ^C_0 D^{\alpha}_t u(x, t) - \Delta u(x, t) = f(x, t),
\quad x \in \Omega, \quad t \in [ 0, T], \\
& u(x, 0) = u_0(x), \quad x \in \Omega, \\
& u(x, t) = 0, \quad x \in \partial \Omega,
\quad t \in [0, T],
\end{split}
\end{equation}
where $\alpha$ is the order of the time-fractional derivative,
$0 < \alpha < 1$, and $\Omega$ is a bounded open domain in $\mathbb{R}^d$,
$ 1 \leq d \leq 3$. The operator $^C_0 D^{\alpha}_t$ is the
Caputo fractional derivative of order $\alpha$
of function $u(x, t)$, defined by
\begin{equation*}
^C_0 D^{\alpha}_t u(x, t) = \frac{1}{\Gamma(1 - \alpha)}
\int_0^t \frac{\partial u(x, s)}{\partial s} \frac{ds}{(t-s)^{\alpha} },
\quad 0<\alpha<1,
\end{equation*}
where $ \Gamma$ denotes the Gamma function. In \cite{yc}, some analytical solutions
of the time-fractional diffusion equation \eqref{eq:01} with a vanishing forcing term
(i.e., $f(x, t) \equiv 0$) are obtained, by applying the finite sine and Laplace
transforms based on the fundamental Mittag--Leffler function.
It is a hard task to search and to compute the exact solution, especially for large time,
due to slow convergence of the series of the Mittag--Leffler function. Therefore,
developing efficient numerical methods is a significant question, and considerable efforts
have been devoted to develop numerical algorithms for this class of problems.
In general, finite difference methods and finite element methods are the most accepted
approaches for solving FPDEs. For instance, in \cite{yc} a practical finite difference/Legendre
spectral method to solve the initial-boundary value time-fractional diffusion problem
\eqref{eq:01}, on a finite domain, is considered. A finite element method for the time-fractional partial
differential equation \eqref{eq:01} on the sense of Riemann--Liouville is introduced
in \cite{NJY} and optimal order error estimates, both in semi-discrete and fully discrete cases,
are obtained. Sidi Ammi and Jamiai have presented also a finite difference and Legendre
spectral method for a time-fractional diffusion-convection equation for image restoration
and a detailed error analysis was carried out \cite{sidiammi1}. In \cite{sidiammi2},
Sidi Ammi and Torres consider a fractional nonlocal thermistor problem and
develop a Galerkin spectral method.  Some error estimates, in different contexts,
are derived, showing that the combination of the backward differentiation in time
and the Galerkin spectral method in space leads, for an enough smooth solution,
to an approximation of exponential convergence in space \cite{sidiammi2}.
Existence and uniqueness of solution for the fractional partial differential equation \eqref{eq:01},
with a left time Riemann--Liouville fractional derivative, is proved in \cite{NJY}
by using the Lax--Milgram Lemma. Here we propose a finite difference method in time
and a finite element method in space to study the numerical solution of the
time-fractional Caputo differential equation \eqref{eq:01}.

The outline of the paper is as follows. In Section~\ref{2},
a finite difference scheme for solving the time-fractional diffusion
equation is proposed, along with an unconditionally stability and convergence analysis.
In Section~\ref{3}, the finite element method is used and error estimates,
in both time and space, are obtained. Then, some numerical tests are presented
in Section~\ref{4}, to verify the accuracy of the given method, comparing the obtained
approximate results with the theoretical/exact ones. Some concluding remarks are given
in Section~\ref{5}. In the analysis of the numerical method that follows,
we assume that problem \eqref{eq:01} has a unique and enough regular solution.


\section{Discretization in time: a finite difference scheme}
\label{2}

In this section we consider the time discretization of \eqref{eq:01}.
Define $A = -\Delta$ and $D(A) = H_0^1(\Omega) \cap H^2(\Omega)$.
Then the system \eqref{eq:01} can be written, in abstract form, as
\begin{align}
\label{eq:05}
& ^C_0 D^{\alpha}_t u(t) + Au(t) = f(t),
\quad 0<t<T, \quad 0< \alpha < 1,\\
\label{eq:06}
& u(0) = u_0.
\end{align}
Let $0=t_0<t_1<\cdots<t_K=T$ be a partition of $ [0, T]$, where
$t_{k}= k \Delta t$, $k=0, 1, \ldots, K$, and $\Delta t= \frac{T}{K}$
is the time step. Following \cite{yc,sidiammi1}, we discretize
the Caputo derivative by a difference approach as follows:
for all $0 \leq k \leq K-1$,
\begin{equation*}
\begin{split}
_0^CD_t^{\alpha}u(t_{k+1})
& =\frac{1}{\Gamma(2-\alpha)\Delta t^\alpha}\sum_{j=0}^k
\left(u(t_{k+1-j})-u(t_{k-j})\right)((j+1)^{1-\alpha}
- j^{1-\alpha}) +\tilde{R}_{k+1},
\end{split}
\end{equation*}
where $\tilde{R}_{k+1}$ is the truncation error satisfying
\begin{equation}
\label{eq:13}
\tilde{R}_{k+1} \leq c_u \Delta t^{2-\alpha}
\end{equation}
and $c_u$ is a constant depending only on $u$.
To continue the construction of the scheme, let us denote
$b_j = (j+1)^{1-\alpha} - j^{1-\alpha}$, $j=0,1,\ldots,k$.
It is easy to verify the following properties for $b_{j}$:
\begin{equation}
\label{eq:20}
\begin{array}{lll}
& b_j > 0, \quad j=0,1,\ldots,k,\\
& 1 = b_0 > b_1 > \cdots > b_k, \quad b_k
\longrightarrow 0 \quad as ~ k \longrightarrow \infty,\\
& \sum_{j=0}^{k}(b_j - b_{j+1}) + b_{k+1}
= (1-b_1) + \sum_{j=1}^{k-1}(b_j - b_{j+1}) + b_k = 1.
\end{array}
\end{equation}
Define the discretized fractional operator $L_{t}^{\alpha}$ by
\begin{equation*}
L_{t}^{\alpha}u(t_{k+1}) =\frac{1}{\alpha_{0}} \left(u(t_{k+1})
- (1-b_1)u(t_k) - \sum_{j=1}^{k-1} (b_j-b_{j+1})u(t_{k-j}) - b_ku(t_0)\right)
\end{equation*}
with
\begin{equation*}
\alpha_0 := \Gamma(2 - \alpha)\Delta t^{\alpha}.
\end{equation*}
Then,
\begin{equation}
\label{eq:14}
\begin{split}
_0^CD_t^{\alpha}u(t_{k+1})
&=\frac{1}{\alpha_0}\left(u(t_{k+1}) - (1-b_1)u(t_k)
- \sum_{j=1}^{k-1} (b_j-b_{j+1})u(t_{k-j}) - b_ku(t_0)\right)
+ \tilde{R}_{k+1}\\
&= L_t^{\alpha}u(t_{k+1}) + \tilde{R}_{k+1}.
\end{split}
\end{equation}
Let $ t=t_{k+1}$. We can write \eqref{eq:05} as
\begin{equation}
\label{eq:07}
L^{\alpha}_t u(t_{k+1}) + Au(t_{k+1})
= f(t_{k+1}) - \tilde{R}_{k+1},
\quad k=0,1,\ldots,K-1.
\end{equation}
Denote $u^k \approx u(t_k)$ as the approximation of $ u(t_k)$.
We define the following time stepping method:
\begin{equation}
\label{eq:21}
L^{\alpha}_t u^{k+1} + Au^{k+1}
= f^{k+1}, \quad k=0,1,\ldots,K-1.
\end{equation}
To complete the semi-discrete problem, we consider the boundary conditions
\begin{equation}
\label{eq:10}
u^{k+1}(x) = 0, \quad k \geq 0,
\quad x \in \partial \Omega,
\end{equation}
and the initial condition
\begin{equation*}
u^0(x) =u_0(x), \quad x \in \Omega.
\end{equation*}
We then obtain an equivalent form to \eqref{eq:21}:
\begin{equation}
\label{eq:08}
u^{k+1} + \alpha_0 Au^{k+1} = \alpha_0 f^{k+1} + (1-b_1)u^k
+ \sum_{j=1}^{k-1} (b_j-b_{j+1})u^{k-j} + b_ku^0, \quad k \geq 1,
\end{equation}
where $ f^{k+1} = f(t_{k+1})$.
For the particular case $k=0$, the scheme becomes
\begin{equation}
\label{eq:09}
u^1 + \alpha_0 A u^1 = \alpha_0 f^1 + u^0.
\end{equation}
If we define the error term $ r^{k+1}$ by
\begin{equation}
\label{eq:12}
r^{k+1} := \alpha_0 \left( _0^CD_t^{\alpha}
u(t_{k+1}) - L_t^{\alpha} u(t_{k+1}) \right),
\end{equation}
then it follows from \eqref{eq:13} and \eqref{eq:14} that
\begin{equation}
\label{eq:15}
|r^{k+1}| = \alpha_0 |r^{k+1}_{\Delta t}|
\leq c_u \Delta t^2.
\end{equation}
Now we define some functional spaces endowed with standard norms
and inner products that will be used in the remaining of the paper:
\begin{equation*}
H^1(\Omega) := \left\{v \in L^2(\Omega), \nabla v \in L^2(\Omega) \right\},
\end{equation*}
\begin{equation*}
H^1_0(\Omega) := \left\{v \in H^1(\Omega), v|_{\partial \Omega}=0 \right\},
\end{equation*}
\begin{equation*}
H^m(\Omega) := \left\{v \in L^2(\Omega), \frac{d^kv}{dx^k} \in L^2(\Omega)
~ \text{ for all positive integer } k \leq m \right\},
\end{equation*}
where $L^2(\Omega)$ is the space of measurable functions whose square
is Lebesgue integrable in $\Omega$. The inner products of $L^2(\Omega)$
and $H^1(\Omega)$ are defined, respectively, by
\begin{equation*}
(u, v) = \int_{\Omega}uvdx,
\quad (u, v)_1 = (u, v) + (\nabla u, \nabla v),
\end{equation*}
while the corresponding norms are given by
\begin{equation*}
\Vert v \Vert = (v, v)^{\frac{1}{2}}, \quad \Vert v \Vert_1
= (v, v)_1^{\frac{1}{2}}, \quad \text{and} \quad
\Vert v \Vert_2 = \Vert v \Vert_{H^2} = \left(
\sum_{k \leq 2}\left\Vert \frac{d^kv}{dx^k}
\right\Vert^2 \right)^{1/2}.
\end{equation*}
The variational weak formulation of equation \eqref{eq:08} subject
to the boundary condition \eqref{eq:10} reads:
find $u^{k+1} \in H^1_0(\Omega)$ such that
\begin{equation}
\label{eq:111}
(u^{k+1}, v) + \alpha_0 (Au^{k+1}, v)
= \alpha_0 (f^{k+1}, v)
+ (1-b_1)(u^k, v) + \sum_{j=1}^{k-1}
(b_j-b_{j+1})(u^{k-j}, v) + b_k(u^0, v)
\end{equation}
$\forall v \in H^1_0(\Omega)$.
Now we consider a stability result of the time discretization
of equations \eqref{eq:01}.

\begin{theorem}
\label{th:01}
Let $ u^k$ be the approximation solution of \eqref{eq:08}. Then,
\begin{equation*}
\Vert u^k \Vert \leq \Vert u^0 \Vert + \alpha_0
k \Vert f \Vert_{L^{\infty}}, \quad k = 1,2,\ldots,K.
\end{equation*}
\end{theorem}

\begin{proof}
The result is proven by mathematical induction.
First, when $k=0$, we have
\begin{equation*}
(I + \alpha_0 A) u^1 = \alpha_0 f^1 + u^0.
\end{equation*}
Then we get
\begin{equation*}
u^1 = (I + \alpha_0 A)^{-1} (\alpha_0 f^1 + u^0).
\end{equation*}
Note that $A$ is a positive definite elliptic operator, the eigenvalues
of $A$ are $\lambda_j = j^2 \pi^2$, $j=1,2,3\ldots$
It follows from the spectral method that the norm
\begin{equation}
\label{eq:18}
\Vert (I + \alpha_0 A)^{-1} \Vert = \sup_{\lambda_j > 0}\vert
(1 + \alpha_0 \lambda_j)^{-1} \vert < 1.
\end{equation}
Hence, by using \eqref{eq:18}, we have
\begin{equation*}
\Vert u^1 \Vert \leq \Vert u^0 + \alpha_0 f^1 \Vert.
\end{equation*}
Then,
\begin{equation*}
\Vert u^1 \Vert \leq \Vert u^0 \Vert
+ \alpha_0 \Vert f \Vert_{L^{\infty}},
\end{equation*}
which suggests the result at the first step.
Suppose now that the following hypothesis holds:
\begin{equation}
\label{eq:19}
\Vert u^j \Vert \leq \Vert u^0 \Vert + j \alpha_0 \Vert f
\Vert_{L^{\infty}}, \quad \forall j=1,2,\ldots,k.
\end{equation}
We begin to prove that $\Vert u^{k+1} \Vert \leq \Vert u^0 \Vert
+ \alpha_0 (k+1)\Vert f \Vert_{L^{\infty}}$.
From \eqref{eq:08}, we have
\begin{equation*}
(I + \alpha_0 A)u^{k+1} = \alpha_0 f^{k+1} + (1-b_1)u^k
+ \sum_{j=1}^{k-1} (b_j-b_{j+1})u^{k-j} + b_ku^0.
\end{equation*}
Hence, by using \eqref{eq:18} and \eqref{eq:19}, one has
\begin{equation*}
\Vert u^{k+1} \Vert \leq  \alpha_0 \Vert f \Vert_{L^{\infty}}
+ \left( (1-b_1) + \sum_{j=1}^{k-1} (b_j-b_{j+1}) + b_k \right)
\left(\Vert u^0 \Vert + \alpha_0 k\Vert f \Vert_{L^{\infty}}\right).
\end{equation*}
Finally, the last equality of \eqref{eq:20} yields
\begin{equation*}
\Vert u^{k+1} \Vert \leq \Vert u^0 \Vert
+ \alpha_0 (k+1)\Vert f \Vert_{L^{\infty}}.
\end{equation*}
The proof is complete.
\end{proof}

We are now ready to prove error estimates in the $L^{2}$
norm for the error $u(t_k) - u^k $ of the approximate solution
$u^k $ of $u$, the exact weak solution $u(t_k)$ of \eqref{eq:21}.

\begin{theorem}
Let $u(t_k)$ and $ u^k$ be the solution of
\eqref{eq:07} and \eqref{eq:21}, respectively. Then,
\begin{equation*}
\Vert u(t_k) - u^k \Vert \leq \frac{c_u}{1 - \alpha}
T^{\alpha} \Delta t^{2 - \alpha},
\end{equation*}
$k = 1,2,\ldots,K$.
\end{theorem}

\begin{proof}
We start by proving the following estimate:
\begin{equation}
\label{eq:24}
\Vert u(t_k) - u^k \Vert \leq c_u b^{-1}_{k-1} \Delta t^2,
\quad k = 1,2,\ldots,K.
\end{equation}
For that we use a standard induction procedure.
Let $\varepsilon^k=u(t_k)-u^k$. For $k=1$, we have, by calling together
\eqref{eq:07}, \eqref{eq:09} and \eqref{eq:12}, that the error equation
is given by
\begin{equation*}
(I + \alpha_0 A) \varepsilon^1 = \varepsilon^0 + r^1.
\end{equation*}
Hence, by using \eqref{eq:18}, we have
\begin{equation*}
\Vert \varepsilon^1 \Vert = \Vert (I + \alpha_0A)^{-1}
r^1 \Vert \leq \Vert r^1 \Vert.
\end{equation*}
With this in mind, and applying \eqref{eq:15}, one obtains that
\begin{equation*}
\Vert u(t_1) - u^1 \Vert \leq c_u b_0^{-1} \Delta t^2.
\end{equation*}
So, \eqref{eq:24} is true for the case $ k= 1$.
Suppose now that \eqref{eq:24} holds for all $k = 1,2,\ldots,K-1$.
By gathering \eqref{eq:07} and \eqref{eq:08}, we have
\begin{equation*}
\varepsilon^{k+1} = (I + \alpha_0A)^{-1} \left( (1-b_1)\varepsilon^k
+ \sum_{j=1}^{k-1} (b_j-b_{j+1})\varepsilon^{k-j}
+ b_k\varepsilon^0  + r^{k+1}\right).
\end{equation*}
It follows that
\begin{equation*}
\Vert \varepsilon^{k+1} \Vert \leq  (1-b_1) \Vert \varepsilon^k \Vert
+ \sum_{j=1}^{k-1} (b_j-b_{j+1}) \Vert \varepsilon^{k-j} \Vert
+ b_k \Vert \varepsilon^0 \Vert  + \Vert r^{k+1} \Vert.
\end{equation*}
Using the induction assumption, and the fact that the sequence
$(b_{j})_{j}$ is decreasing, we obtain that
\begin{equation*}
\begin{split}
\Vert \varepsilon^{k+1} \Vert
& \leq \left( (1-b_1) b_{k-1}^{-1}
+ \sum_{j=1}^{k-1} (b_j-b_{j+1}) b_{k-j-1}^{-1} \right)
c_u \Delta t^2 + c_u \Delta t^2\\
& \leq \left( (1-b_1) + \sum_{j=1}^{k-1} (b_j-b_{j+1})
+ b_k \right) c_u b_k^{-1}\Delta t^2.
\end{split}
\end{equation*}
Taking into account \eqref{eq:20} in the above inequality,
it follows that
\begin{equation*}
\Vert \varepsilon^{k+1} \Vert \leq c_u b_k^{-1} \Delta t^2.
\end{equation*}
The auxiliary estimate \eqref{eq:24} is then established.
One can easily verify that $k^{- \alpha} b_{k-1}^{-1} \leq \frac{1}{1- \alpha}$
and $k \Delta t \leq T$, $ k=1,2,\ldots,K$. Thus,
\begin{equation*}
\begin{split}
\Vert u(t_k) - u^k \Vert
& \leq c_u k^{- \alpha} b^{-1}_{k-1} k^{\alpha} \Delta t^2
\leq \frac{c_u}{1- \alpha}(k \Delta t)^{\alpha} \Delta t^{2 - \alpha}\\
&  \leq \frac{c_u}{1- \alpha}T^{\alpha} \Delta t^{2 - \alpha}
\end{split}
\end{equation*}
and the proof is complete.
\end{proof}

In the coming section, we consider the space discretization
of \eqref{eq:01}.


\section{Discretization in space: a finite element scheme}
\label{3}

The variational formulation of \eqref{eq:01} consists
to find $ u(t) \in H_0^1(\Omega)$, such that
\begin{equation*}
(_0^CD_t^{\alpha}u(t),v) + (\nabla u(t), \nabla v)
= (f(t), v), \quad \forall v \in H_0^1(\Omega).
\end{equation*}
More precisely, let $ 0=x_0<x_1<\cdots<x_N=1$ be an arbitrary space partition
of $\Omega = [ 0, 1] \subset \mathbb{R}$ and let $h=\max_{i}(x_{i+1}-x_{i})$.
The set $\Omega$ can be a set of $\mathbb{R}^{n}$, $1 \leq n \leq 3$.
Let $ S_{h} \subseteq H_{0}^{1}(\Omega)$ be a family of a finite element space
consisting of piecewise linear continuous functions defined by
$$
S_h=\left\{ v_h / v_h
\text{ is a piecewise linear continuous function on }
\Omega \right\}.
$$
Now consider the finite element method as follows:
find $ u_h(t) \in S_h$, such that
\begin{equation*}
(_0^CD_t^{\alpha}u_h(t),\phi) + (\nabla u_h(t), \nabla \phi)
= (f(t), \phi), \quad \forall \phi \in S_h.
\end{equation*}
Denote $A_h = - \Delta_h~ :~ S_h~\rightarrow~S_h$, which satisfies
\begin{equation*}
(A_h u_h, \phi) = (\nabla u_h, \nabla \phi), \quad \forall \phi \in S_h.
\end{equation*}
Let $ P_h~:~ H^{1}_{_{0}}(\Omega)~\rightarrow~S_h$ be the standard
$L_2$ projection operator via the orthogonal relation
\begin{equation*}
(P_h v, \phi) = (v , \phi), \quad \forall \phi \in S_h, \quad v \in L^{2}(\Omega)
\end{equation*}
and $R_h~:~H_0^1(\Omega)~\rightarrow~S_h$ be the elliptic
or the Ritz projection defined by
\begin{equation*}
(\nabla (R_h u), \nabla \phi) = (\nabla u, \nabla \phi),
\quad \forall \phi \in S_h.
\end{equation*}
We can write \eqref{eq:05} into abstract form as
\begin{equation}
\label{eq:26}
^C_0 D^{\alpha}_t u_h(t) + A_hu_h(t) = P_hf(t), \quad 0<t<T.
\end{equation}
Denote by $u_h^{j}$ the approximation of $u(x, t_j)$.
We define the following time stepping method:
\begin{equation}
\label{eq:27}
L^{\alpha}_t u^{k+1}_h + A_hu^{k+1}_h
= f^{k+1}, \quad k=0,1,\ldots,K-1.
\end{equation}
Now consider the finite element discretization of problem \eqref{eq:111}
as follows: find $u^{k+1}_h \in S_h$, such that for all $v_h \in S_h$
\begin{equation}
\label{eq:112}
B_h(u^{k+1}_h, v_h) = F_h(v_h),
\quad \forall v_h \in S_h,
\end{equation}
where the bilinear form $B_h(\cdot, \cdot)$ is defined by
$$
B_h(u^{k+1}_h, v_h) = (u^{k+1}_h, v_h)
+ \alpha_0 (A_hu^{k+1}_h, v_h)
$$
and the functional $F_h(v_h)$ is given by
$$
F_h(v_h) = \alpha_0 (f^{k+1}_h, v_h) + (1-b_1)(u^k_h, v_h)
+ \sum_{j=1}^{k-1} (b_j-b_{j+1})(u^{k-j}_h, v_h) + b_k(u^0_h, v_h).
$$
Then, under enough regularity of the exact solution $u$,
the following error estimate holds.

\begin{theorem}
Let $u(t_k)$ and $ u_h^k$ be the solution of \eqref{eq:05} and \eqref{eq:27},
respectively. Assume that $u \in H^1\left(H^2(\Omega) 
\cap H^1_0(\Omega), [0, T]\right)$. 
Then, the following inequality holds:
\begin{equation}
\label{eq:28}
\Vert u( t_k) - u_h^k \Vert
\leq O\left(\Delta t^{2-\alpha} + h^2\right).
\end{equation}
\end{theorem}

\begin{proof}
Let $ \varepsilon_h^k = u_h^k - u(t_k)$. We write
\begin{equation*}
\begin{split}
\varepsilon_h^k
&= u_h^k - R_hu(t_k) + R_hu(t_k) - u(t_k)\\
&= \theta_h^k + \rho_h^k, \quad k = 0,1, \ldots,
\end{split}
\end{equation*}
where $ \theta_h^k = u_h^k - R_hu(t_k)$ and
$\rho_h^k = R_hu(t_k) - u(t_k)$.
We make use of the following inequality that
can be find in \cite{AEL,AMYP}:
$$
\Vert R_h v - v \Vert \leq Ch^2 \Vert v \Vert_2.
$$
By the error estimate of the Ritz projection, we have
\begin{equation}
\label{eq:29}
\Vert \rho_h^k \Vert = \Vert R_hu(t_k) - u(t_k)
\Vert \leq Ch^2 \Vert u(t_k) \Vert_{2}.
\end{equation}
In order to bound $\theta_h^k$, we use the error
equation obtained from \eqref{eq:27}:
$$
L^{\alpha}_t \theta^{k+1}_h + A_h \theta^{k+1}_h
= L^{\alpha}_t u^{k+1}_h + A_hu^{k+1}_h
- L^{\alpha}_t R_hu(t_{k+1}) - A_hR_hu(t_{k+1}).
$$
From $ \Delta_h R_h = P_h \Delta$ (see \cite{VT}),
we hence obtain that
\begin{equation*}
\begin{split}
L^{\alpha}_t \theta^{k+1}_h + A_h \theta^{k+1}_h
&= P_h(f^{k+1}) - P_hA u(t_{k+1}) - L^{\alpha}_t R_hu(t_{k+1})\\
&= P_h(f^{k+1}) + (P_h-R_h)L^{\alpha}_t u(t_{k+1})
-P_hL^{\alpha}_t u(t_{k+1}) - P_hA u(t_{k+1})\\
&= P_h(f^{k+1}) + (P_h-R_h)L^{\alpha}_t u(t_{k+1})
- P_h(L^{\alpha}_t u(t_{k+1}) + A u(t_{k+1}))\\
&= P_h(f^{k+1}) + P_h(I-R_h)L^{\alpha}_t u(t_{k+1})
- P_h(f^{k+1} - \tilde{R}_{k+1})\\
&= P_h\left((I-R_h)L^{\alpha}_t u(t_{k+1})
+ \tilde{R}_{k+1}\right)\\
&= P_h \left(\omega_h^{k+1}\right),
\end{split}
\end{equation*}
where
$\omega_h^{k+1} = \sigma_h^{k+1} + \tilde{R}_{k+1}$
and
$\sigma_h^{k+1} = (I-R_h)L^{\alpha}_t u(t_{k+1})$.
Thus, we get
$$
L^{\alpha}_t \theta^{k+1}_h + A_h \theta^{k+1}_h
= P_h\left(\sigma_h^{k+1} + \tilde{R}_{k+1}\right).
$$
Using the stability result of Theorem~\ref{th:01}, we obtain that
$$
\Vert \theta_h^k\Vert \leq \Vert \theta^0\Vert+\alpha_0
k\left\Vert P_h\left(\sigma_h^{k+1}
+ \tilde{R}_{k+1}\right)\right\Vert_{\infty}.
$$
Hence, by using \eqref{eq:29}, we have
\begin{equation*}
\begin{split}
\alpha_0\Vert \sigma_h^{k+1}\Vert
&= \alpha_0\Vert R_hL_t^{\alpha}u(t_{k+1})-L_t^{\alpha}u(t_{k+1})\Vert\\
& \leq \alpha_0Ch^2\Vert L_t^{\alpha}u(t_{k+1})\Vert_2\\
& \leq Ch^2 \left\Vert u(t_{k+1})-(1-b_1)u(t_k)
-\sum_{j=1}^{k-1}(b_j-b_{j+1})u(t_{k-j})-b_{k}u(t_0)\right\Vert_2\\
& \leq Ch^2 \left( \Vert u(t_{k+1})\Vert_2+(1-b_1)\Vert u(t_k)\Vert_2
+\sum_{j=1}^{k-1}(b_j-b_{j+1})\Vert u(t_{k-j})\Vert_2+b_k\Vert u(t_0)\Vert_2\right)\\
& \leq Ch^2\left(1+(1-b_1)+\sum_{j=1}^{k-1}(b_j-b_{j+1})+b_k\right)
\max_{0 \leq j \leq k+1} \left\Vert u(t_j)\right\Vert_2\\
& \leq 2 C h^2 \max_{0 \leq j \leq k+1} \left\Vert u(t_j)\right\Vert_2.
\end{split}
\end{equation*}
Keeping in mind that $\Vert \tilde{R}_{k+1}\Vert \leq c_u \Delta t^{2-\alpha}$,
we obtain that
\begin{equation*}
\Vert \theta_h^k \Vert \leq \alpha_0 k c_u \Delta t^{2-\alpha}
+ 2k C h^2 \max_{0 \leq j \leq k+1} \left\Vert u(t_j)\right\Vert_2.
\end{equation*}
Together with this estimate, we get
\begin{equation*}
\begin{split}
\Vert \varepsilon_h^k\Vert
& \leq \Vert \rho_h^k \Vert + \Vert \theta_h^k \Vert\\
& \leq \alpha_0kc_u \Delta t^{2-\alpha} +2kCh^2\max_{0\leq j\leq k+1}
\Vert u(t_j)\Vert_2 + Ch^2\Vert u(t_k)\Vert_2\\
& \leq \alpha_0kc_u \Delta t^{2-\alpha}
+ ch^2\max_{0\leq j\leq k+1}\Vert u(t_j)\Vert_2.
\end{split}
\end{equation*}
Hence, \eqref{eq:28} is proved.
\end{proof}


\section{Numerical validation}
\label{4}

For completeness, our implementation is briefly described here.


\subsection{Implementation}
\label{sec:41}

Considering problem \eqref{eq:112}, we express the function $u^{k+1}_h$
in terms of the finite piecewise linear elements, tent-line, global
interpolation functions $\phi_j(x)$, $j = 0, 1, \ldots, N$,
\begin{equation}
\label{eq:113}
u^{k+1}_h(x) = \sum_{j=0}^{N} u_j(t_{k+1}) \phi_j(x),
\end{equation}
where $u_j(t_{k+1}) = u(x_j, t_{k+1})$ are unknowns of the numerical solution
and $\phi_j$ are the global interpolation functions satisfying
the cardinal interpolation property
$$
\phi_j(x_i) = \delta_{ij} \quad \forall\,\,
i, j \in \left\{0, 1, \ldots, N\right\}
$$
with $\delta_{ij}$ the Kronecker-delta symbol.
By combining \eqref{eq:112} and \eqref{eq:113}, and taking
into account the homogeneous Dirichlet boundary condition
$u^{k+1}_0 = u^{k+1}_N = 0$, we obtain the discrete system
\begin{equation}
\label{eq:114}
(\mathbb{M} + \alpha_0\mathbb{S}) \mathbb{U}^{k+1}
= \alpha_0\mathbb{M}\textbf{f}^{k+1} + (1-b_1)\mathbb{M}\mathbb{U}^k
+ \mathbb{M}\sum_{j=1}^{k-1}(b_j-b_{j+1})\mathbb{U}^{k-j}
+b_k\mathbb{M}\mathbb{U}^0,
\end{equation}
where
$$
\mathbb{U}^k = [u^k_1,u^k_2,\ldots,u^k_{N-1}],
\,\, \textbf{f}^k = [f^k_1,f^k_2,\ldots,f^k_{N+1}],
$$
\begin{equation*}
\mathbb{M}= (M_{ij})_{1 \leq i, j \leq N}, M_{ij}
= (\phi_i,\phi_j), \quad S_{ij}
= (\nabla \phi_i, \nabla \phi_j).
\end{equation*}
Since the matrix $\mathbb{M} + \alpha_0\mathbb{S}$ is symmetric
positive definite, one can choose, for example,
the conjugate gradient method to solve \eqref{eq:114}.


\subsection{Numerical results}
\label{sec:42}

Now we present two numerical approximation
examples to confirm our theoretical statements.
The main purpose is to check the convergence behavior
of the discrete solution with respect to the time step
$\Delta t$ and the space step $\Delta x$ used in the computations.

\begin{example}
\label{ex:01}
Consider the time-fractional partial differential equation
\begin{align*}
& ^C_0 D^{\alpha}_t u(x, t) - \Delta u(x, t)
= f(x, t), \quad t \in [0,T], \quad 0<x<1,\\
& u(x, 0) = u_0(x), \quad 0<x<1,\\
& u(0, t) = u(1, t) = 0, \quad t \in [0, T].
\end{align*}
The right hand side $f$ and initial condition are selected as
$$
f(x,t)=\frac{2}{\Gamma(3-\alpha)} t^{2-\alpha}
\sin(2\pi x) + 4\pi^2 t^2 \sin(2\pi x), ~~~~~~ u_0(x) = 0.
$$
It is verified that the exact solution to the problem is
$$
u(x, t) = t^2 \sin(2\pi x).
$$
The numerical results have been given by choosing $\Delta x=0.001$, 
$T=1$, and $\Delta t=0.01$, where $N=T/\Delta t$.
Let $u^n$ denote the approximate solution, $u(t_n)$ the exact solution, 
and $\varepsilon^n$ the error at $t=t_n$, that is, $\varepsilon^n=u^n -u(t_n)$. 
Then we obtain Table~\ref{Table01} with the exact solution, the approximate 
solution, and the error for $\alpha = 0.1,0.5,0.9$. We plot the exact solution, 
the approximate solution, and the error, for $\alpha = 0.1,0.5,0.9$, 
in Figures~\ref{fig:ex01}, \ref{fig:ex011}, and \ref{fig:ex0111}.
\begin{figure}
\subfigure[The exact solution $u$.]{\label{Fig:01}\includegraphics[scale=0.50]{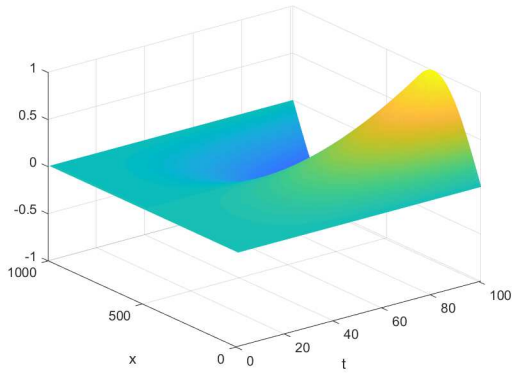}}
\subfigure[The approximate solution $u^n$.]{\label{Fig:02}\includegraphics[scale=0.50]{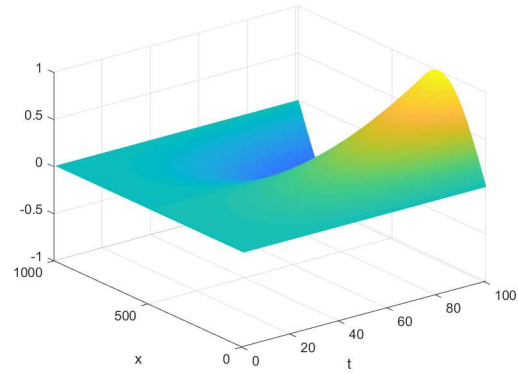}}
\subfigure[The error between $u$ and $u^n$.]{\label{Fig:03}\includegraphics[scale=0.5]{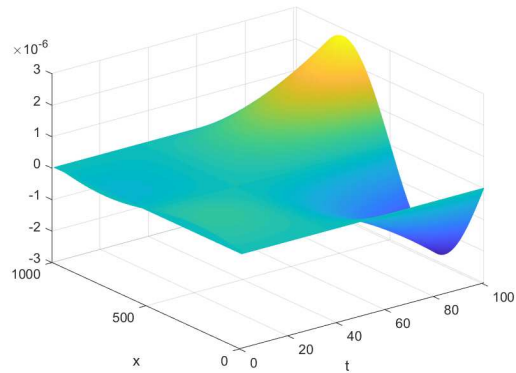}}
\caption{Results for the problem of Example~\ref{ex:01} with $\alpha = 0.1$.}
\label{fig:ex01}
\end{figure}
\begin{figure}
\subfigure[The approximate solution $u^n$.]{\includegraphics[scale=0.50]{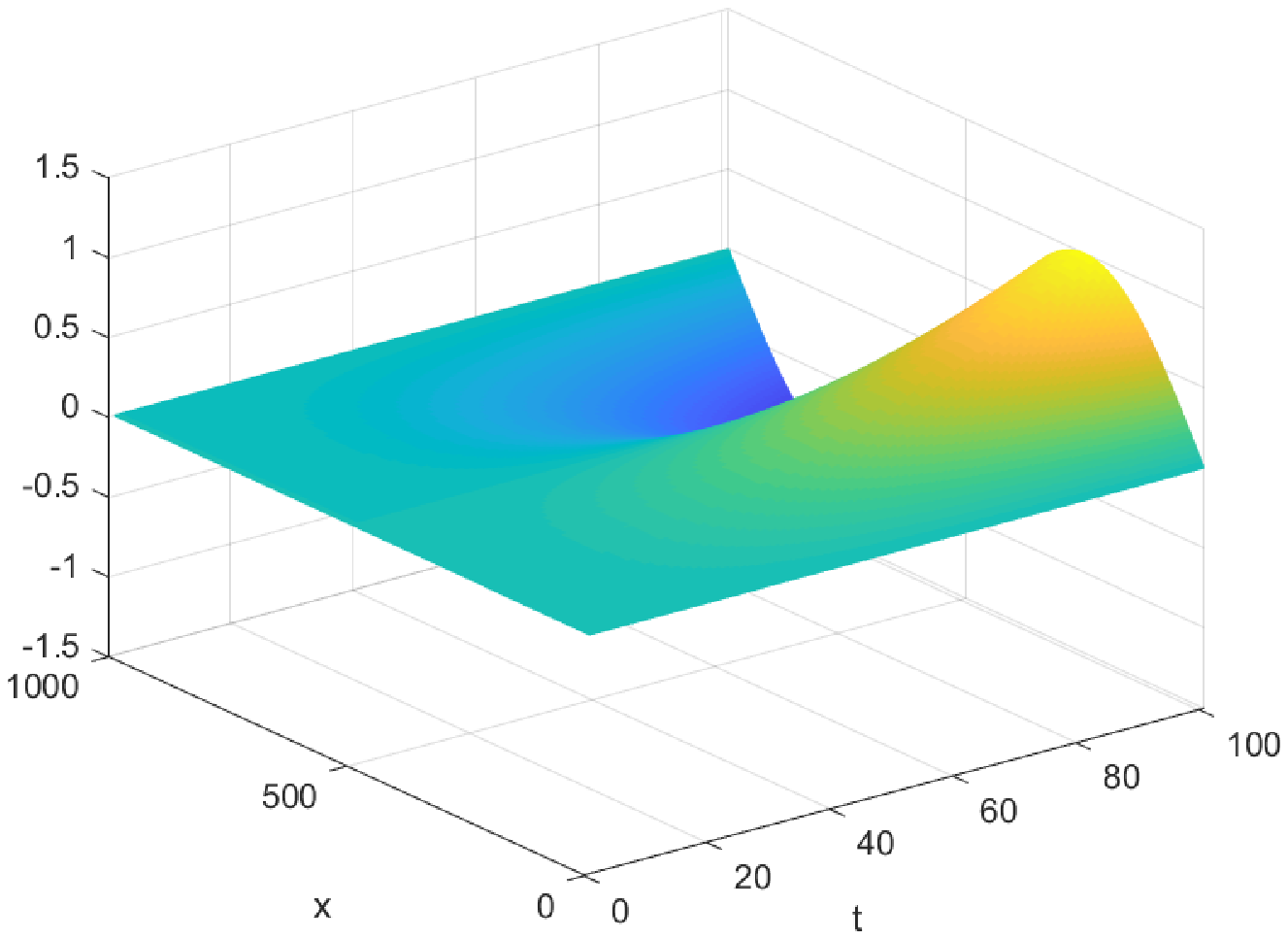}}
\subfigure[The error between $u$ and $u^n$.]{\includegraphics[scale=0.5]{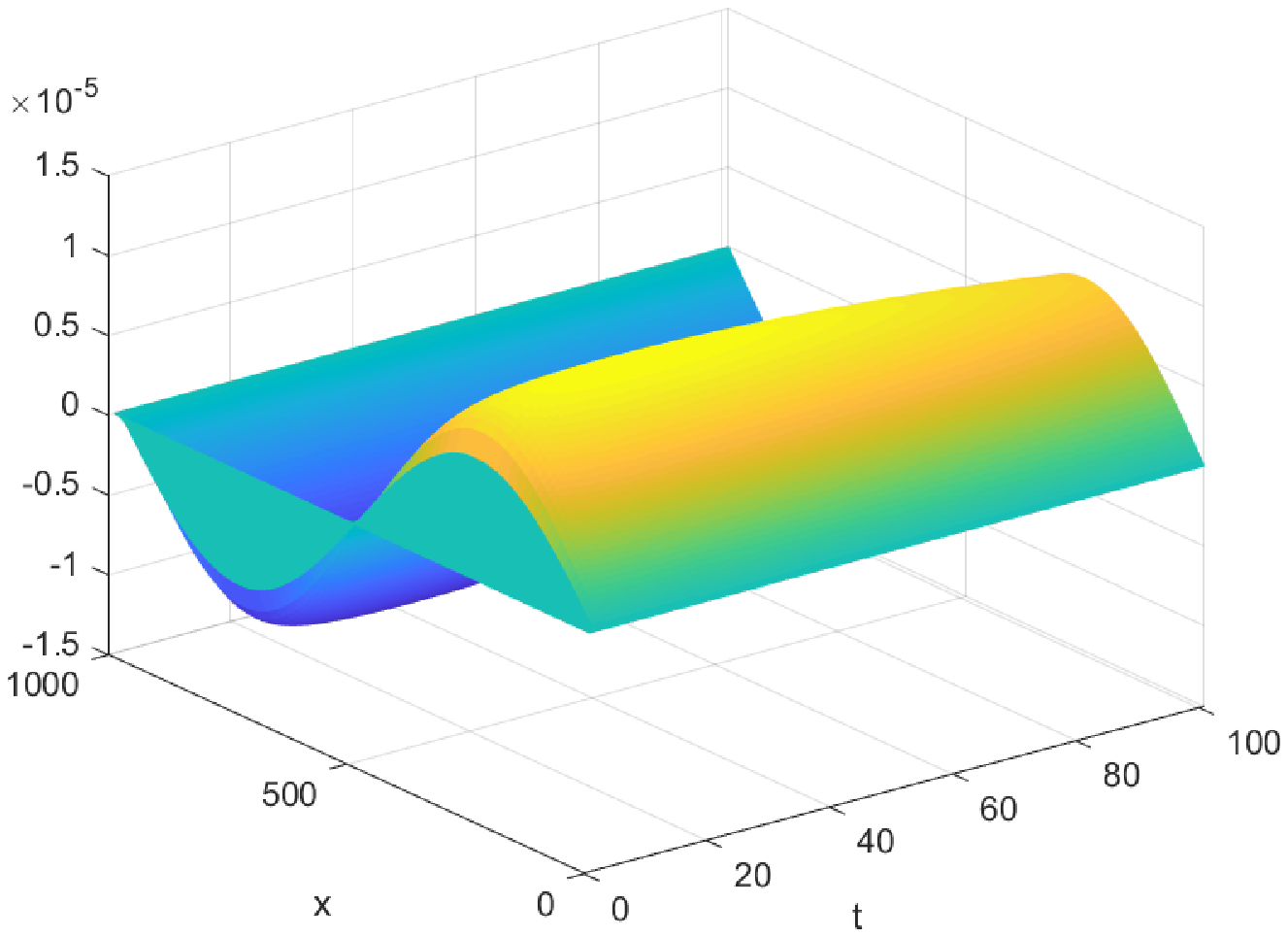}}
\caption{Results for the problem of Example~\ref{ex:01} with $\alpha = 0.5$.}
\label{fig:ex011}
\end{figure}
\begin{figure}
\subfigure[The approximate solution $u^n$.]{\includegraphics[scale=0.50]{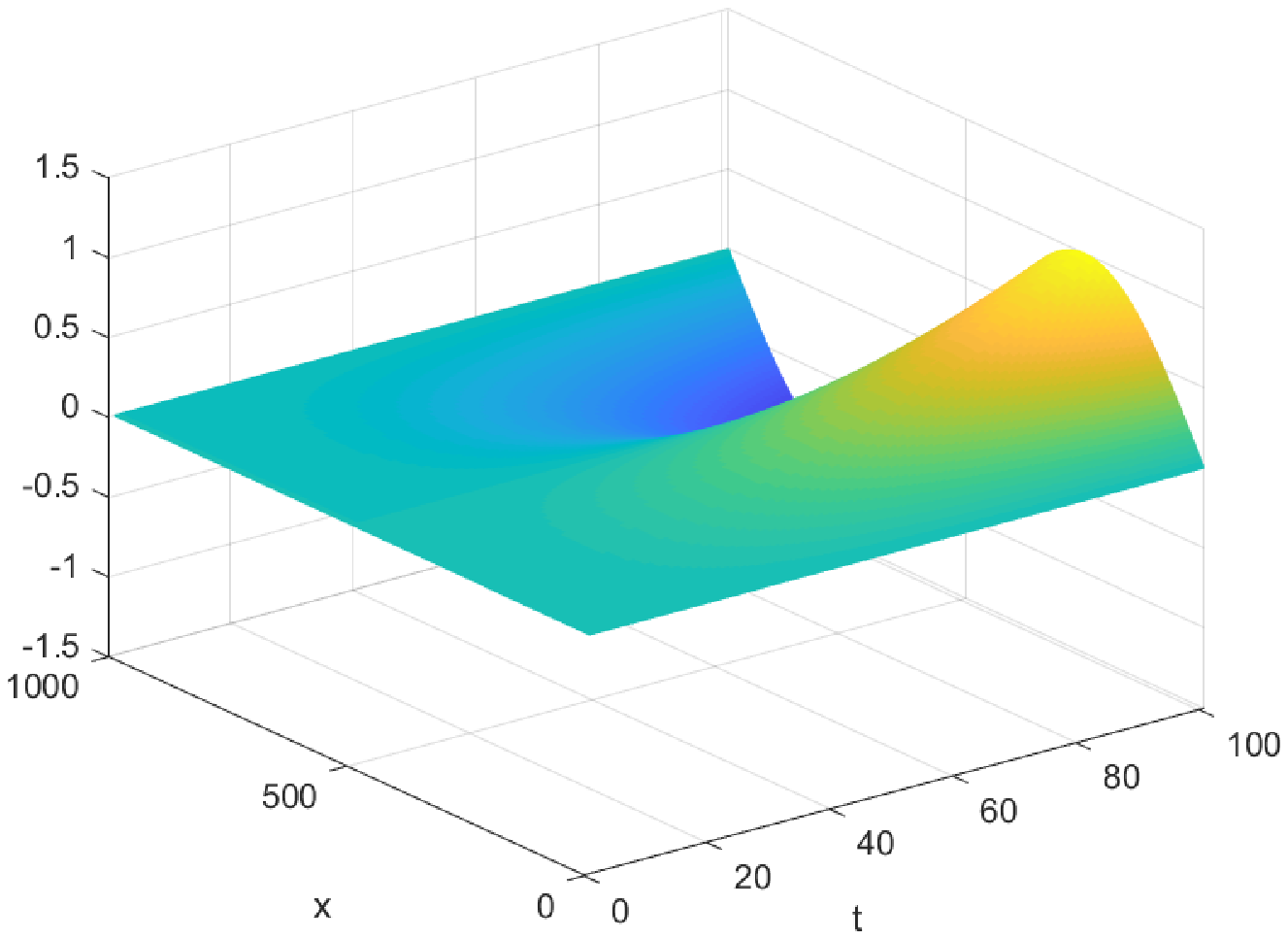}}
\subfigure[The error between $u$ and $u^n$.]{\label{Fig:07}\includegraphics[scale=0.5]{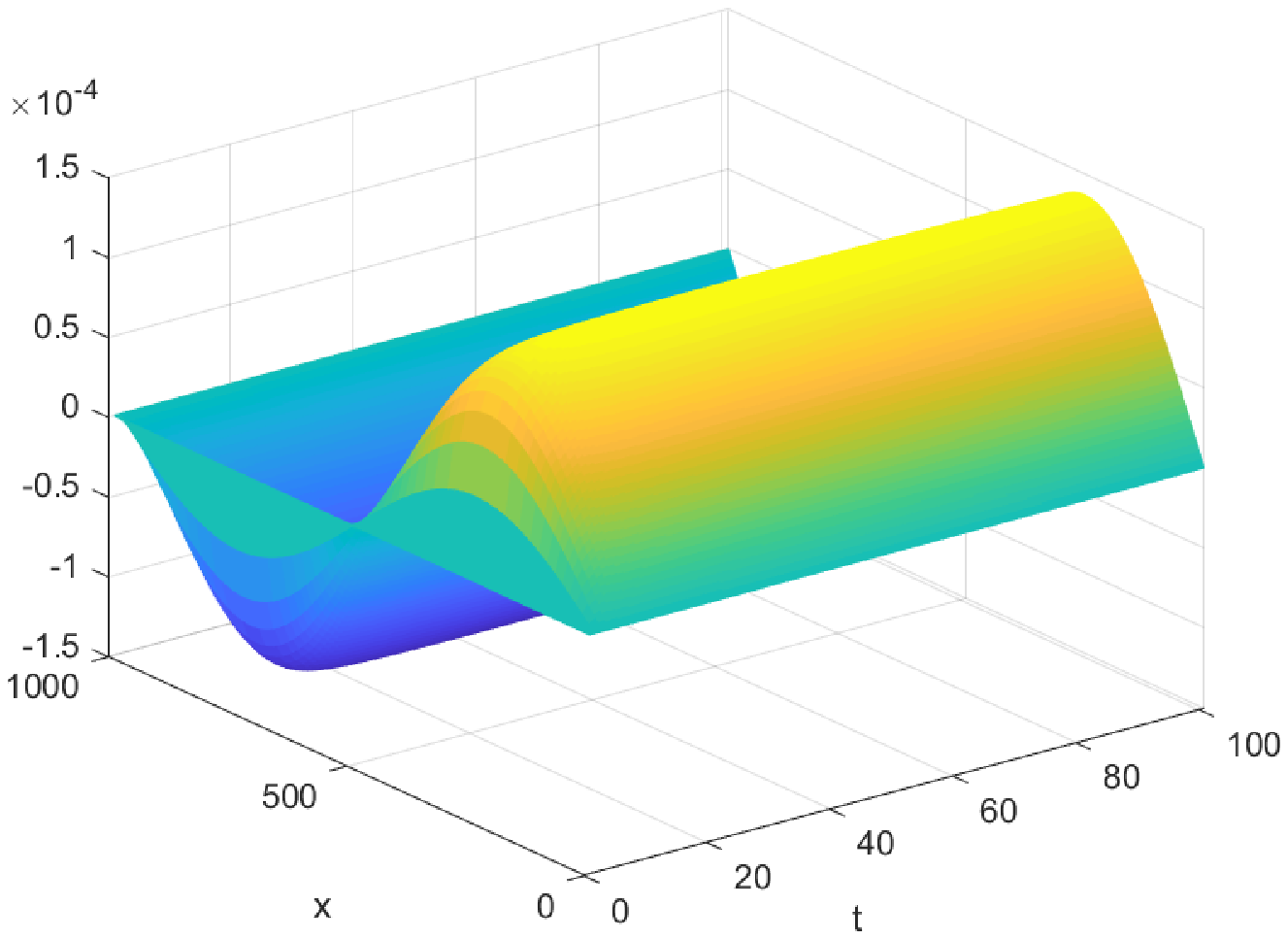}}
\caption{Results for the problem of Example~\ref{ex:01} with $\alpha = 0.9$.}
\label{fig:ex0111}
\end{figure}
\begin{table}[ht]
\center
\caption{Results for the problem of Example~\ref{ex:01} with $\alpha = 0.1,0.5,0.9$.\label{Table01}}
\begin{tabular}{|c|c|c|c|c|} \cline{1-5} \hline
\multicolumn{1}{ |c| }{$x_i$} & {exact sol} & {$\alpha$} & {approximate sol} & {error} \\ \cline{1-5}
\multicolumn{1}{ |c|  }{\multirow{3}{*}{$x_{100}$ }  }  &  &
\multicolumn{1}{ |c| }{0.1} & {$0.571093$} & {$2.10^{-6}$} \\ \cline{3-5}
\multicolumn{1}{ |c }{}     &
\multicolumn{1}{ |c| }{$0.571095$} & {0.5} & {$0.571100$} & {$-5.10^{-6}$} \\ \cline{3-5}
\multicolumn{1}{ |c| }{}                &        &
\multicolumn{1}{ |c| }{0.9} & {$0.571175$} & {$-8.10^{-5}$} \\ \cline{1-5}
\multicolumn{1}{ |c|  }{\multirow{3}{*}{$x_{300}$}}  &  &
\multicolumn{1}{ |c| }{0.1} & {$0.934012$} & {$3.10^{-6}$}  \\ \cline{3-5}
\multicolumn{1}{ |c }{}     &
\multicolumn{1}{ |c| }{0.934015} & {0.5} & {$0.934023$} & {$-8.10^{-6}$}      \\ \cline{3-5}
\multicolumn{1}{ |c| }{}                 &        &
\multicolumn{1}{ |c| }{0.9} & {$0.934145$} & {$-1,3.10^{-4}$}     \\ \cline{1-5}
\multicolumn{1}{ |c|  }{\multirow{3}{*}{$x_{500}$}}  &  &
\multicolumn{1}{ |c| }{0.1} & {$0.006158$} & {$0$}      \\ \cline{3-5}
\multicolumn{1}{ |c }{}                &
\multicolumn{1}{ |c| }{0.006158} & {0.5} & {$0.006158$} & {$0$}      \\ \cline{3-5}
\multicolumn{1}{ |c| }{}                &        &
\multicolumn{1}{ |c| }{0.9} & {$0.006159$} & {$-1.10^{-6}$}     \\ \cline{1-5}
\multicolumn{1}{ |c|  }{\multirow{3}{*}{$x_{700}$}}  &  &
\multicolumn{1}{ |c| }{0.1} & {$-0.930207$} & {$-3.10^{-6}$}      \\ \cline{3-5}
\multicolumn{1}{ |c }{}                &
\multicolumn{1}{ |c| }{-0.930209} & {0.5} & {$-0.930217$} & {$8.10^{-6}$}      \\ \cline{3-5}
\multicolumn{1}{ |c| }{}                &        &
\multicolumn{1}{ |c| }{0.9} & {$-0.930339$} & {$1,3.10^{-4}$}     \\ \cline{1-5}
\multicolumn{1}{ |c|  }{\multirow{3}{*}{$x_{900}$}}  &  &
\multicolumn{1}{ |c| }{0.1} & {$-0.581057$} & {$-2.10^{-6}$}      \\ \cline{3-5}
\multicolumn{1}{ |c }{}                &
\multicolumn{1}{ |c| }{-0.581059} & {0.5} & {$-0.581064$} & {$5.10^{-6}$}      \\ \cline{3-5}
\multicolumn{1}{ |c| }{}                &        &
\multicolumn{1}{ |c| }{0.9} & {$-0.581140$} & {$8,1.10^{-6}$}     \\ \cline{1-5}
\end{tabular}
\end{table}
\end{example}

\begin{example}
\label{ex:02}
Consider the time-fractional partial differential equation
\begin{align*}
& ^C_0 D^{\alpha}_t u(x, t) - \Delta u(x, t)
= f(x, t), \quad t \in [0,T], \quad 0<x<1,\\
& u(x, 0) = u_0(x), \quad 0<x<1,\\
& u(0, t) = u(1, t) = 0, \quad t \in [0, T],
\end{align*}
with the forcing term and initial condition given by
$$
f(x,t)=\frac{1}{\Gamma(1-\alpha)} \int_0^t \pi (t-s)^{-\alpha}
\cos(\pi s) \sin(\pi x) ds - \pi^2 \sin(\pi t) \sin(\pi x), 
\quad u_0(x) = 0.
$$
The exact solution is
$$
u(x, t) = \sin(\pi t) \sin(\pi x).
$$
In this second example, we choose $\alpha=0.2$, $\Delta x=0.01$,
$T=1$, $\Delta t=0.01$, and $N=T/\Delta t$.
Let $u^n$ denote the approximate solution, $u(t_n)$ the exact solution,
and $\varepsilon^n=u^n -u(t_n)$ the error at $t=t_n$. Figures~\ref{Fig:04}
and \ref{Fig:05} illustrate, respectively, the exact solution and the approximate solution
at $t_N=1$. Figure~\ref{Fig:06} presents a plot of the error at $t_N=1$.
\begin{figure}
\subfigure[The exact solution $u$.]{\label{Fig:04}\includegraphics[scale=0.15]{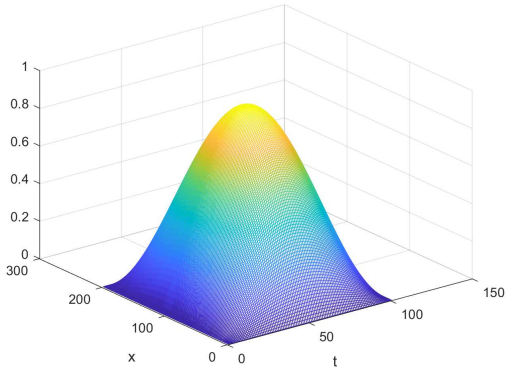}}
\subfigure[The approximate solution $u^n$.]{\label{Fig:05}\includegraphics[scale=0.15]{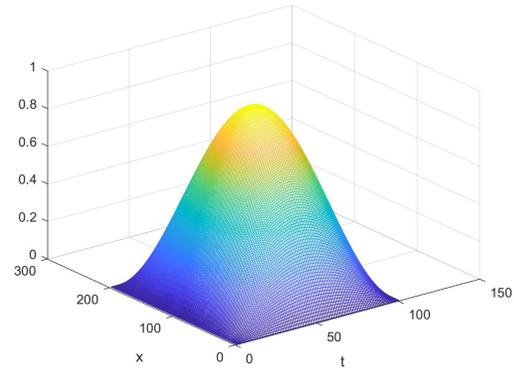}}
\subfigure[The error between $u$ and $u^n$.]{\label{Fig:06}\includegraphics[scale=0.2]{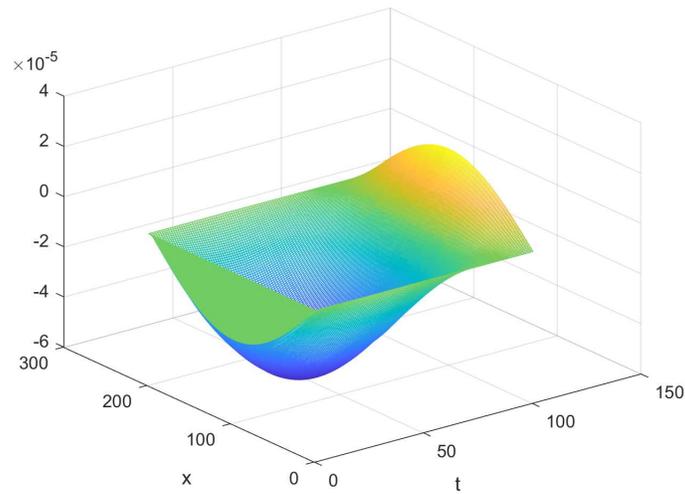}}
\caption{Results for the problem of Example~\ref{ex:02}.}
\label{fig:ex02}
\end{figure}
\end{example}


\section{Conclusion}
\label{5}

We have investigated a finite element method to Caputo
time-fractional diffusion partial differential equations.
A stability analysis is carried out and a convergent estimate is analyzed.
We obtain error estimates in the $L_2$-norm between the exact solution
and the approximate solutions in the fully discrete case.
Two numerical examples are implemented and the numerical results are shown
to be consistent with the theoretical results.


\section*{Acknowledgements}

The authors were supported by the Center for Research and Development in
Mathematics and Applications (CIDMA) of University of Aveiro, through 
\emph{Funda\c{c}\~ao para a Ci\^encia e a Tecnologia} (FCT), 
within project UID/MAT/04106/2019. They are very grateful
to two anonymous referees, for careful reviews of their paper, 
and for the comments, corrections, and suggestions, which 
substantially helped them to improve the quality of the paper. 



\medskip



\begin{thebibliography}{99}

\bibitem{AEL}
K. Adolfsson, M. Enelund\ and\ S. Larsson,
Adaptive discretization of fractional order viscoelasticity using sparse time history,
Comput. Methods Appl. Mech. Engrg. {\bf 193} (2004), no.~42-44, 4567--4590.

\bibitem{MR3787674}
R. P. Agarwal, D. Baleanu, J. J. Nieto, D. F. M. Torres\ and\ Y. Zhou,
A survey on fuzzy fractional differential and optimal control nonlocal evolution equations,
J. Comput. Appl. Math. {\bf 339} (2018), 3--29.
{\tt  arXiv:1709.07766}

\bibitem{MR3822307}
R. Almeida, D. Tavares\ and\ D. F. M. Torres,
{\it The variable-order fractional calculus of variations},
SpringerBriefs in Applied Sciences and Technology, Springer, Cham, 2019.
{\tt arXiv:1805.00720}

\bibitem{AMYP}
F. Amblard, A. C. Maggs, B. Yurke, A. N. Pargellis\ and\ S. Leibler,
Subdiffusion and anomalous local viscoelasticity in actin networks,
Phys Rev Lett. {\bf 77} (1996), no.~21, 4470--4473.

\bibitem{6}
C. \c{C}elik\ and\ M. Duman,
Crank-Nicolson method for the fractional diffusion equation
with the Riesz fractional derivative,
J. Comput. Phys. {\bf 231} (2012), no.~4, 1743--1750.

\bibitem{MR3830855}
M. Dehghan\ and\ M. Safarpoor,
Application of the dual reciprocity boundary integral equation
approach to solve fourth-order time-fractional partial differential equations,
Int. J. Comput. Math. {\bf 95} (2018), no.~10, 2066--2081.

\bibitem{8}
W. Deng,
Finite element method for the space and time fractional Fokker-Planck equation,
SIAM J. Numer. Anal. {\bf 47} (2008/09), no.~1, 204--226.

\bibitem{10}
A. Elsaid,
The variational iteration method for solving Riesz fractional
partial differential equations,
Comput. Math. Appl. {\bf 60} (2010), no.~7, 1940--1947.

\bibitem{11}
V. J. Ervin\ and\ J. P. Roop,
Variational formulation for the stationary fractional advection dispersion equation,
Numer. Methods Partial Differential Equations {\bf 22} (2006), no.~3, 558--576.

\bibitem{NJY}
N. J. Ford, J. Xiao\ and\ Y. Yan,
A finite element method for time fractional partial differential equations,
Fract. Calc. Appl. Anal. {\bf 14} (2011), no.~3, 454--474.

\bibitem{8bis}
R. Gorenflo, Y. Luchko\ and\ F. Mainardi,
Wright functions as scale-invariant solutions of the diffusion-wave equation,
J. Comput. Appl. Math. {\bf 118} (2000), no.~1-2, 175--191.

\bibitem{9}
R. Gorenflo, F. Mainardi, D. Moretti\ and\ P. Paradisi,
Time fractional diffusion: a discrete random walk approach,
Nonlinear Dynam. {\bf 29} (2002), no.~1-4, 129--143.

\bibitem{10bis}
F. Huang\ and\ F. Liu,
The time fractional diffusion equation and the advection-dispersion equation,
ANZIAM J. {\bf 46} (2005), no.~3, 317--330.

\bibitem{MR3859787}
B. A. Jacobs\ and\ C. Harley,
Application of nonlinear time-fractional partial differential equations
to image processing via hybrid Laplace transform method,
J. Math. {\bf 2018} (2018), Art. ID 8924547, 9~pp.

\bibitem{XJCJ}
X. Li\ and\ C. Xu,
Existence and uniqueness of the weak solution of the space-time fractional
diffusion equation and a spectral method approximation,
Commun. Comput. Phys. {\bf 8} (2010), no.~5, 1016--1051.

\bibitem{yc}
Y. Lin\ and\ C. Xu,
Finite difference/spectral approximations for the time-fractional diffusion equation,
J. Comput. Phys. {\bf 225} (2007), no.~2, 1533--1552.

\bibitem{11bis}
F. Liu, V. V. Anh, I. Turner\ and\ P. Zhuang,
Time fractional advection-dispersion equation,
J. Appl. Math. Comput. {\bf 13} (2003), no.~1-2, 233--245.

\bibitem{26}
F. Liu, P. Zhuang, V. Anh, I. Turner\ and\ K. Burrage,
Stability and convergence of the difference methods
for the space-time fractional advection-diffusion equation,
Appl. Math. Comput. {\bf 191} (2007), no.~1, 12--20.

\bibitem{30}
M. M. Meerschaert\ and\ C. Tadjeran,
Finite difference approximations for fractional advection-dispersion flow equations,
J. Comput. Appl. Math. {\bf 172} (2004), no.~1, 65--77.

\bibitem{34}
J. Qin\ and\ T. Wang,
A compact locally one-dimensional finite difference method
for nonhomogeneous parabolic differential equations,
Int. J. Numer. Methods Biomed. Eng. {\bf 27} (2011), no.~1, 128--142.

\bibitem{36}
S. Saha Ray\ and\ R. K. Bera,
An approximate solution of a nonlinear fractional
differential equation by Adomian decomposition method,
Appl. Math. Comput. {\bf 167} (2005), no.~1, 561--571.

\bibitem{MR3854267}
A. B. Salati, M. Shamsi\ and\ D. F. M. Torres,
Direct transcription methods based on fractional integral approximation
formulas for solving nonlinear fractional optimal control problems,
Commun. Nonlinear Sci. Numer. Simul. {\bf 67} (2019), 334--350.
{\tt arXiv:1805.06537}

\bibitem{MR3845039}
M. Sarboland,
Numerical solution of time fractional partial differential equations
using multiquadric quasi-interpolation scheme,
Eur. J. Comput. Mech. {\bf 27} (2018), no.~2, 89--108.

\bibitem{sidiammi1}
M. R. Sidi Ammi\ and\ I. Jamiai,
Finite difference and Legendre spectral method for a
time-fractional diffusion-convection equation for image restoration,
Discrete Contin. Dyn. Syst. Ser. S {\bf 11} (2018), no.~1, 103--117.

\bibitem{sidiammi2}
M. R. Sidi Ammi\ and\ D. F. M. Torres,
Galerkin spectral method for the fractional nonlocal thermistor problem,
Comput. Math. Appl. {\bf 73} (2017), no.~6, 1077--1086.
{\tt arXiv:1605.07804}

\bibitem{ZSW}
Z. Sun\ and\ X. Wu,
A fully discrete difference scheme for a diffusion-wave system,
Appl. Numer. Math. {\bf 56} (2006), no.~2, 193--209.

\bibitem{VT}
V. Thom\'{e}e,
{\it Galerkin finite element methods for parabolic problems},
second edition, Springer Series in Computational Mathematics,
25, Springer-Verlag, Berlin, 2006.

\end{thebibliography}
\end{document}